\documentclass{amsart}

\usepackage{mathscinet,mathrsfs}

\usepackage{graphicx,epsfig,psfrag}

\newtheorem{theorem}{Theorem}[section]
\newtheorem{lemma}[theorem]{Lemma}

\newtheorem{corollary}[theorem]{Corollary}

\theoremstyle{definition}

\theoremstyle{remark}

\numberwithin{equation}{section}

\def\R{{\mathbb R}}
\def\N{{\mathbb N}}

\def\A{{\mathscr A}}

\newcommand{\be}{\begin{equation}}
\newcommand{\bi}{\begin{itemize}}
\newcommand{\ei}{\end{itemize}}
\newcommand{\ee}{\end{equation}}
\newcommand{\ba}{\begin{eqnarray}}
\newcommand{\bs}{\begin{eqnarray*}}
\newcommand{\ea}{\end{eqnarray}}
\newcommand{\es}{\end{eqnarray*}}

\begin{document}

\title{On the continuity of global attractors}

\author{Luan T.\ Hoang}
\address{Department of Mathematics and Statistics,
Texas Tech University,
Box 41042,
Lubbock, TX 79409-1042. USA.}
\email{luan.hoang@ttu.edu}
%\thanks{RL was supported in part by XXX.}

\author{Eric J.\ Olson}
\address{Department of
  Mathematics/084, University of Nevada, Reno, NV 89557. USA.}
\email{ejolson@unr.edu}
%\thanks{EJO was supported by EP/G007470/1.}

\author{James C.\ Robinson}
\address{Mathematics Institute, University of Warwick, Coventry CV4 7AL. UK.}
\email{j.c.robinson@warwick.ac.uk}
\thanks{JCR was supported by an EPSRC Leadership Fellowship EP/G007470/1, which supported the time spent in Warwick by LTH and EJO}

%    General info

%\date{January 1, 2001 and, in revised form, June 22, 2001.}



\begin{abstract}
Let $\Lambda$ be a complete metric space, and let $\{S_\lambda(\cdot):\ \lambda\in\Lambda\}$ be a parametrised family of semigroups with global attractors $\A_\lambda$. We assume that there exists a fixed bounded set $D$ such that $\A_\lambda\subset D$ for every $\lambda\in\Lambda$. By viewing the attractors as the limit as $t\to\infty$ of the sets $S_\lambda(t)D$, we give simple proofs of the equivalence of `equi-attraction' to continuity (when this convergence is uniform in $\lambda$) and show that the attractors $\A_\lambda$ are continuous in $\lambda$ at a residual set of parameters in the sense of Baire Category (when the convergence is only pointwise).
\end{abstract}

\maketitle

\section{Global attractors}

The global attractor of a dynamical system is the unique compact invariant set that attracts the trajectories starting in any bounded set at a uniform rate. Introduced by Billotti \& LaSalle \cite{BLS},
they have been the subject of much research since the mid-1980s, and form the central topic of a number of monographs, including Babin \& Vishik \cite{BV}, Hale \cite{Hale}, Ladyzhenskaya \cite{Ladyz}, Robinson \cite{JCR}, and Temam \cite{Temam}.

The standard theory incorporates existence results \cite{BLS}, upper semicontinuity \cite{HLR}, and bounds on the attractor dimension \cite{CFT}. Global attractors exist for many infinite-dimensional models \cite{Temam}, with familiar low-dimensional ODE models such as the Lorenz equations providing a testing ground for the general theory \cite{DG}.

While upper semicontinuity with respect to perturbations is easy to prove, lower semicontinuity (and hence full continuity) is more delicate, requiring structural assumptions on the attractor or the assumption of a uniform attraction rate. However, Babin \& Pilyugin \cite{BPY} proved that the global attractor of a parametrised set of semigroups is continuous at a residual set of parameters, by taking advantage of the known upper semicontinuity and then using the fact that upper semicontinuous functions are continuous on a residual set.

Here we reprove results on equi-attraction and residual continuity in a more direct way, which also serves to demonstrate more clearly why these results are true. Given equi-attraction the attractor is the uniform limit of a sequence of continuous functions, and hence continuous (the converse requires a generalised version of Dini's Theorem); more generally, it is the pointwise limit of a sequence of continuous functions, i.e.\ a `Baire one' function, and therefore the set of continuity points forms a residual set.

\section{Semigroups and attractors}

A semigroup $\{S(t)\}_{t\ge0}$ on a complete metric space $(X,d)$ is a collection of maps $S(t):X\to X$ such that
\begin{itemize}
\item[(S1)] $S(t)0={\rm id}$;
\item[(S2)] $S(t+s)=S(t)S(s)=S(s)S(t)$ for all $t,s\ge0$; and
\item[(S3)] $S(t)x$ is continuous in $x$ and $t$.
\end{itemize}
A compact set $\A\subset X$ is the \emph{global attractor} for $S(\cdot)$ if
\begin{itemize}
\item[(A1)] $S(t)\A=\A$ for all $t\in\R$; and
\item[(A2)] for any bounded set $B$, $\rho_X(S(t)B,\A)\to0$ as $t\to\infty$, where $\rho_X$ is the semi-distance $
\rho_X(A,B)=\sup_{a\in A}\inf_{b\in B}d(a,b)$.
\end{itemize}
When such a set exists it is unique, the maximal compact invariant set, and the minimal closed set that satisfies (A2).

\section{Upper and lower semicontinuity}

Let $\Lambda$ be a complete metric space and $S_\lambda(\cdot)$ a parametrised family of semigroups on $X$. Suppose that
\begin{itemize}
\item[(L1)] $S_\lambda(\cdot)$ has a global attractor $\A_\lambda$ for every $\lambda\in\Lambda$;
\item[(L2)] there is a bounded subset $D$ of $X$ such that $\A_\lambda\subset D$ for every $\lambda\in\Lambda$; and
\item[(L3)] for $t>0$, $S_\lambda(t)x$ is continuous in $\lambda$, uniformly for $x$ in bounded subsets of $X$.
\end{itemize}
We can strengthen (L2) and weaken (L3) by replacing `bounded' by `compact' to yield conditions (L2$'$) and (L3$'$). A wide range of dissipative systems with parameters satisfy these assumptions, for example the 2D Navier--Stokes equations, the scalar Kuramoto--Sivashinsky equation, reaction-diffusion equations, and the Lorenz equations, all of which are covered in \cite{Temam}.

Under these mild assumptions it is easy to show that $\A_\lambda$ is upper semicontinuous,
 $$
 \rho_X(\A_\lambda,\A_{\lambda_0})\to0\qquad\mbox{as}\quad\lambda\to\lambda_0
 $$
 see \cite{BV,BPY,Chu,Hale,HLR,JCR,Temam}, for example. However, lower semicontinuity, that is
 $$
  \rho_X(\A_{\lambda_0},\A_{\lambda})\to0\qquad\mbox{as}\quad\lambda\to\lambda_0,
 $$
requires more: either structural conditions on the attractor $\A_{\lambda_0}$ ($\A_{\lambda_0}$ is the closure of the unstable manifolds of a finite number of hyperbolic equilibria, see Hale \& Raugel\ \cite{HR}, Stuart \& Humphries \cite{SH}, or Robinson \cite{JCR}) or the `equi-attraction' hypothesis of Li \& Kloeden \cite{KL} (see Section \ref{EA}). As a result, the continuity of attractors,
$$
\lim_{\lambda\to\lambda_0}d_{\rm H}(\A_\lambda,\A_{\lambda_0})=0,
$$
where
\begin{equation}\label{sHd}
d_{\rm H}(A,B)=\max(\rho_X(A,B),\rho_X(B,A))
\end{equation}
is the symmetric Hausdorff distance, is only known under restrictive conditions.

In this paper we view $\A_\lambda$ as a function from $\Lambda$ into the space of closed bounded subsets of $X$, given as the limit of the continuous functions $\overline{S_\lambda(t_n)D}$ (see Lemma \ref{Dncts}). Indeed, note that given any set $D\supseteq\A_\lambda$ it follows from the invariance of the attractor (A1) that
$$
\overline{S_\lambda(t)D}\supseteq S_\lambda(t)D\supseteq S_\lambda(t)\A_\lambda=\A_\lambda\qquad\mbox{for every}\quad t>0,
$$
and so the the attraction property of the attractor in (A2) implies that
\begin{equation}\label{limit}
d_{\rm H}(\overline{S_\lambda(t)D},\A_\lambda)\to0\qquad\mbox{as}\qquad t\to\infty.
\end{equation}

Uniform convergence (with respect to $\lambda$) in (\ref{limit}) is essentially the `equi-attraction' introduced in \cite{KL}, and thus clearly related to continuity of the limiting function $\A_\lambda$ (Section \ref{EA}). Given only pointwise ($\lambda$-by-$\lambda$) convergence in (\ref{limit}) we can still use the result  from the theory of Baire Category that the pointwise limit of continuous functions (a `Baire one function') is continuous at a residual set to guarantee that $\A_\lambda$ is continuous in $\lambda$ on a residual subset of $\Lambda$ (Section \ref{BC}).

For both results the following simple lemma is fundamental. We let $CB(X)$ be the collection of all closed and bounded subsets of $X$, and use the symmetric Hausdorff distance $d_{\rm H}$ defined in (\ref{sHd}) as the metric on this space.

\begin{lemma}\label{Dncts}
Suppose that $D$ is bounded and that {\rm(L3)} holds. Then for any $t>0$ the map $\lambda\mapsto \overline{S_\lambda(t)D}$ is continuous from $\Lambda$ into $CB(X)$. The same is true if $D$ is compact and {\rm(L3$'$)} holds.
\end{lemma}

\begin{proof}
  Given $t>0$, $\lambda_0\in\Lambda$, and $\epsilon>0$, (L3) ensures that there exists a $\delta>0$ such that $d_\Lambda(\lambda_0,\lambda)<\delta$ implies that $d_X(S_\lambda(t)x,S_{\lambda_0}(t)x)<\epsilon$ for every $x\in D$. It follows that
  $$
  \rho_X(S_\lambda(t)D,S_{\lambda_0}(t)D)<\epsilon\qquad\mbox{and}\qquad\rho_X(S_{\lambda_0}(t)D,S_{\lambda}(t)D)<\epsilon,
  $$
  and so
   $$
  \rho_X(\overline{S_\lambda(t)D},\overline{S_{\lambda_0}(t)D})\le\epsilon\qquad\mbox{and}\qquad\rho_X(\overline{S_{\lambda_0}(t)D},\overline{S_{\lambda}(t)D})\le\epsilon,
  $$
  from which $d_{\rm H}(\overline{S_\lambda(t)D},\overline{S_{\lambda_0}(t)D})\le\epsilon$ as required.\end{proof}

\section{Uniform convergence: continuity and equi-attraction}\label{EA}

First we give a simple proof of the results in \cite{KL} on the equivalence between equi-attraction and continuity. In our framework these follow from two classical results: the continuity of the uniform limit of a sequence of continuous functions and Dini's Theorem in an abstract formulation.

Li \& Kloeden require (L1), (L2'), a time-uniform version of (L3') (see Corollary \ref{Cor}), and in addition an `equi-dissipative' assumption that there exists a bounded absorbing set $K$ such that
\begin{equation}\label{equiD}
S_\lambda(t)B\subset K\qquad\mbox{for every}\quad t\ge t_B,
\end{equation}
where $t_B$ does not depend on $\lambda$. They then show that when $\Lambda$ is compact, continuity of $\A_\lambda$ in $\lambda$ is equivalent to equi-attraction,
  \begin{equation}\label{unifKL}
  \lim_{t\to\infty}\sup_{\lambda\in\Lambda}\rho_X(S_\lambda(t)D,{\mathscr A}_\lambda)\to0\qquad\mbox{as}\quad t\to\infty.
  \end{equation}

  We now give our version of Dini's Theorem.

\begin{theorem}\label{Dini}
 For each $n\in\N$ let $f_n:K\to Y$ be a continuous map, where $K$ is a compact metric space and $Y$ is any metric space. If $f$ is continuous and is the monotonic pointwise limit of $f_n$, i.e.\ for every $x\in K$
 $$
 d_Y(f_n(x),f(x))\to0\quad\mbox{as}\quad n\to\infty\quad\mbox{and}\quad d_Y(f_{n+1}(x),f(x))\le d_Y(f_n(x),f(x))
 $$
 then $f_n$ converges uniformly to $f$.
\end{theorem}

\begin{proof}
  Given $\epsilon>0$ define
  $$
  E_n=\{x\in K:\ d_Y(f_n(x),f(x))<\epsilon\}.
  $$
  Since $f_n$ and $f$ are both continuous, $E_n$ is open and non-decreasing. Since $K$ is compact and $\cup_{n=1}^\infty E_n$ provides an open cover of $K$, there exists an $N(\epsilon)$ such that $K=\cup_{n=1}^N E_n$, and so $d_Y(f_n(x),f(x))<\epsilon$ for all $x\in K$ for all $n\ge N(\epsilon)$.
\end{proof}

Our first result relates continuity to a slightly weakened form of equi-attraction through sequences. We remark that our proof allows us to dispense with the `equi-dissipative' assumption (\ref{equiD}) of \cite{KL}.

\begin{theorem}\label{cty}
  Assume {\rm(L1)} and {\rm(L2--3)} or {\rm(L2$'$--3$'$)}. If there exist $t_n\to\infty$ such that
  \begin{equation}\label{unif}
  \lim_{n\to\infty}\sup_{\lambda\in\Lambda}\rho_X(S_\lambda(t_n)D,{\mathscr A}_\lambda)\to0\qquad\mbox{as}\quad t\to\infty,
  \end{equation}
  then $\A_\lambda$ is continuous in $\lambda$ for all $\lambda\in\Lambda$. Conversely, if $\Lambda$ is compact then
%
 % and the convergence in {\rm(L3)} or {\rm(L3$'$)} is also uniform for $t$ in compact subsets of %$[0,\infty)$, then
 % , if there exists a bounded set $D\supset\A_\lambda$ such that
%  $$
%  S_\lambda(T)D\subset D\qquad\mbox{for all}\quad\lambda\in\Lambda
%  $$
%  then
  continuity of $\A_\lambda$ for all $\lambda\in\Lambda$ implies that there exist $t_n\to\infty$ such that (\ref{unif}) holds.
\end{theorem}

\begin{proof}
Lemma \ref{Dncts} guarantees that $\lambda\mapsto \overline{S_\lambda(t_n)D}$ is continuous for each $n$, and we have already observed in (\ref{limit}) that
$$
d_{\rm H}(\overline{S_\lambda(t_n)D},{\mathscr A}_\lambda)\to0\qquad\mbox{as}\quad n\to\infty.
$$
 Therefore $\A_\lambda$ is the uniform limit of the continuous functions $\overline{S_\lambda(t_n)D}$ and so is continuous itself.

  For the converse, let $D_1=\{x\in X:\ \rho_X(x,D)<1\}$. For each $\lambda_0\in\Lambda$ it follows from (A2) and (L2) that there exists a time $t(\lambda_0)\in\N$ such that $S_{\lambda_0}(t)D_1\subseteq D$ for all $t\ge t({\lambda_0})$. It follows from (L3) that there exists an $\epsilon(\lambda_0)>0$ such that
  \begin{equation}\label{first}
  S_\lambda(t(\lambda_0))D_1\subset D_1
  \end{equation}
  for every $\lambda$ with $d_\Lambda(\lambda,\lambda_0)<\varepsilon(\lambda_0)$.

      Since $\Lambda$ is compact
       $$
       \Lambda=\bigcup_{\lambda\in\Lambda}B_{\epsilon(\lambda)}(\lambda)=\bigcup_{k=1}^N B_{\epsilon(\lambda_k)}(\lambda_k)
       $$
       for some $N\in\N$ and $\lambda_k\in\Lambda$. If $T=\prod_{k=1}^N t(\lambda_k)$ then
       \begin{equation}\label{second}
       S_\lambda(T)D_1\subseteq D_1\qquad\mbox{for every}\quad\lambda\in\Lambda,
       \end{equation}
       since any $\lambda\in\Lambda$ is contained in $B_{\epsilon(\lambda_k)}(\lambda_k)$ for some $k$, and $T=mt(\lambda_k)$ for some $m\in\N$, from which (\ref{second}) follows by applying $S_\lambda(t(\lambda_k))$ repeatedly ($m-1$ times) to both sides of (\ref{first}).

       It follows from (\ref{second}) that for every $\lambda\in\Lambda$, $\overline{S_\lambda(nT)D_1}$ is a decreasing sequence of sets, and hence the convergence of $\overline{S_\lambda(nT)D_1}$ to $\A_\lambda$, ensured by (\ref{limit}), is in fact monotonic in the sense of our Theorem \ref{Dini}. Uniform convergence now follows, and finally the fact that $D\subseteq D_1$ yields
       $$
       d_{\rm H}(S_\lambda(nT)D,\A_\lambda)\le d_{\rm H}(S_\lambda(nT)D_1,\A_\lambda)\to0
       $$
       uniformly in $\lambda$ as $n\to\infty$.\end{proof}

 With additional uniform continuity assumptions we can readily show that continuity implies equi-attraction in the sense of \cite{KL}. We give one version of this result.

\begin{corollary}\label{Cor} Suppose that {\rm(L1--3)} hold and that $\Lambda$ is compact. Assume in addition that $S_\lambda(t)x$ is continuous in $x$, uniformly in $\lambda$ and for $x$ in bounded subsets of $X$ and $t\in[0,T]$ for any $T>0$. Then continuity of $\A_\lambda$ implies (\ref{unifKL}).
\end{corollary}

\begin{proof}
Since $\A_\lambda\subset D$ and $\A_\lambda$ is invariant, given any $\epsilon>0$, by assumption there exists a $\delta>0$ such that
\begin{equation}\label{morecty}
d_X(d,\A_\lambda)<\delta\qquad\Rightarrow\qquad d_X(S_\lambda(\tau)d,\A_\lambda)<\epsilon
\end{equation}
 any $d\in X$, for all $\lambda\in\Lambda$ and $\tau\in(0,T)$. Choose $n_0$ sufficiently large that
$$
d_{\rm H}(S_\lambda(nT)D,\A_\lambda)<\delta\qquad\mbox{for all}\quad n\ge n_0;
$$
now for any $t\in(nT,(n+1)T)$, $n\ge n_0$, we can write $t=nT+\tau$ for some $\tau\in(0,T)$, and it follows from (\ref{morecty}) that
$$
d_{\rm H}(S_\lambda(t)D,\A_\lambda)<\epsilon\qquad t\ge n_0T,
$$
with the convergence uniform in $\lambda$ as required.\end{proof}

\section{Pointwise convergence and residual continuity}\label{BC}

When the rate of attraction to $\A_\lambda$ is not uniform in $\lambda$ we nevertheless have the convergence in (\ref{limit}) for each $\lambda$. In general, therefore, one can view the attractor as the `pointwise' ($\lambda$-by-$\lambda$) limit of the sequence of continuous functions $\overline{S_\lambda(t)D}$. In the case of a sequence of continuous real functions, their pointwise limit is a `Baire one function', and is continuous on a residual set. The same ideas in a more abstract setting yield continuity of $\A_\lambda$ on a residual subset of $\Lambda$.

We use the following abstract result, characterising the continuity of `Baire one' functions, whose proof (which we include for completeness) is an easy variant of that given by Oxtoby \cite{Oxtoby}. A result in the same general setting as here can be found as Theorem 48.5 in Munkres \cite{Munk}. Recall that a set is \emph{nowhere dense} if its closure contains no open sets, and a set is \emph{residual} if its complement is the countable union of nowhere dense sets. Any residual set is dense.

\begin{theorem}\label{BCT}
  For each $n\in\N$ let $f_n:\Lambda\to Y$ be a continuous map, where $\Lambda$ is a complete metric space and $Y$ is any metric space. If $f$ is the pointwise limit of $f_n$, i.e.\ $f(\lambda)=\lim_{n\to\infty}f_n(\lambda)$ for each $\lambda\in\Lambda$ (and the limit exists), then the points of continuity of $f$ form a residual subset of $\Lambda$.
\end{theorem}

Before the proof we make the following observation: if $U$ and $V$ are open subsets of $\Lambda$ with $V\subset\overline U$, then $U\cap V\neq\emptyset$. Otherwise $V^c$, the complement of $V$ in $\Lambda$, is a closed set containing $U$, and since $\overline U$ is the intersection of all closed sets that contain $U$, $\overline U\subset V^c$, a contradiction.

\begin{proof}
  We show that for any $\delta>0$ the set
  $$
  F_\delta=\{\lambda_0\in\Lambda:\ \lim_{\epsilon\to0}\sup_{d_\Lambda(\lambda,\lambda_0)\le\epsilon}d_Y(f(\lambda),f(\lambda_0))\ge3\delta\}
  $$
  is nowhere dense. From this it follows that
  $$
  \cup_{n\in\N}F_{1/n}=\{\mbox{discontinuity points of }f\}
  $$
  is nowhere dense, and so the set of continuity points is residual.

  To show that $F_\delta$ is nowhere dense, i.e.\ that its closure contains no open set, let
  $$
  E_n(\delta)=\{\lambda\in\Lambda:\ \sup_{i,j\ge n}d_Y(f_i(\lambda),f_j(\lambda))\le\delta\}.
  $$
  Note that $E_n$ is closed, $E_{n+1}\supset E_n$, and $\Lambda=\cup_{n=0}^\infty E_n$. Choose any open set $U\subset\Lambda$, and consider $\overline U=\cup_{n=0}^\infty \overline U\cap E_n$. Since $\overline U$ is a complete metric space, it follows from the Baire Category Theorem that there exists an $n$ such that $\overline U\cap E_n$ contains an open set $V'$. From the remark before the proof, $V:=V'\cap U$ is an open subset of $\overline U\cap E_n$ that is in addition a subset of $U$.

  Since $V\subset E_n$, it follows that $d_Y(f_i(\lambda),f_j(\lambda))\le\eta$ for all $\lambda\in V$ and $i,j\ge n$. Fixing $i=n$ and letting $j\to\infty$ it follows that
  $$
  d_Y(f_n(\lambda),f(\lambda))\le\eta\qquad\mbox{for all}\quad\lambda\in V.
  $$
  Now, since $f_n(\lambda)$ is continuous in $\lambda$, for any $\lambda_0\in V$ there is a neighbourhood $N(\lambda_0)\subset V$ such that
  $$
  d_Y(f_n(\lambda),f_n(\lambda_0))\le\eta\qquad\mbox{for all}\quad\lambda\in N(\lambda_0).
  $$
  Thus by the triangle inequality
  $$
  d_Y(f(\lambda_0),f(\lambda))\le 3\eta\qquad\mbox{for all}\quad\lambda\in N(\lambda_0).
  $$
  It follows that no element of $N(\lambda_0)$ belongs to $F_\delta$, which implies, since $N(\lambda_0)\subset V\subset U$ that $U$ contains an open set that is not contained in $F_\delta$. This shows that $F_\delta$ is nowhere dense, which concludes the proof.
\end{proof}

\begin{theorem}\label{main}
  Under assumptions {\rm(L1--3)} above -- or {\rm(L1)}, {\rm(L2$'$)}, and {\rm(L3$'$)} -- $\A_\lambda$ is continuous in $\lambda$ for all $\lambda_0$ in a residual subset of $\Lambda$. In particular the set of continuity points of $\A_\lambda$ is dense in $\Lambda$.
\end{theorem}

%Before giving the proof we note that in particular
%: given any $\lambda_0\in\Lambda$ and $\epsilon>0$, there is a $\lambda'$ with $d_\Lambda(\lambda',\lambda_0)<\epsilon$ at which $\A_\lambda$ is continuous.% at $\lambda'$.

\begin{proof}
We showed in Lemma \ref{Dncts} that for every $t>0$ the map $\lambda\mapsto \overline{S_\lambda(n)D}$ is continuous from $\Lambda$ into $BC(X)$, and observed in (\ref{limit}) the pointwise convergence
  $$
 d_{\rm H}(\overline{S_\lambda(t)D},\A_\lambda)\to0\qquad\mbox{as}\quad t\to\infty.
  $$
 The result follows immediately from Theorem \ref{BCT}, setting $f_n(\lambda)=\overline{S_\lambda(n)D}$ and $f(\lambda)=A_\lambda$ for every $\lambda\in\Lambda$.
\end{proof}
%
%\begin{theorem}
%   Let $X$ be a compact topological space and suppose that $G_n:X\to 2^Y$ satisfies (i) $G_{n+1}(x)\subseteq G_n(x)$; (ii) $d_{\rm H}(G_n(x),\mathscr G(x))\to0$ as $n\to\infty$ for each $x\in X$; and (iii) $G:X\to 2^Y$ is continuous. Then $G_n$ converges uniformly to $G$.
%\end{theorem}
%
%\begin{proof}
%  Given $\epsilon>0$ let $E_n$ be the set of those $x\in X$ such that
%  $$
%  d_{\rm H}(G_n(x),G(x))<\epsilon.
%  $$
%  Since $G_n$ and $G$ are both continuous, $E_n$ is open. Since .... !!?
%\end{proof}

Residual continuity results also hold for the pullback attractors \cite{CLR} and uniform attractors \cite{CV} that occur in non-autonomous systems. We will discuss these results in the context of the two-dimensional Navier--Stokes equations in a future paper.

\end{document}